\documentclass[twoside,11pt,preprint]{article}

\usepackage{amsmath}
\usepackage{amssymb}
\usepackage{bm}
\usepackage{mathrsfs}

\usepackage{algorithmic}
\usepackage{algorithm}

\usepackage{color}

\usepackage{url}

\usepackage{supertabular}

\newtheorem{thm}{Theorem}
\newtheorem{cor}[thm]{Corollary}

\DeclareMathAlphabet{\mathsfsl}{OT1}{cmss}{m}{sl}

\renewcommand{\phi}{\varphi}

\newcommand{\R}{\mathbb{R}}

\newcommand{\argmin}{\operatorname*{arg\; min}}

\newcommand{\Expect}{\operatorname{\mathbb{E}}}

\def\reals{\mathbb{R}}

\def\reals{\mathbb{R}}

\def\b0{\mathbf{0}}

\def\calI{\mathcal{I}}

\def\sign{\mathrm{sign}}

\usepackage{jmlr2e}

\ShortHeadings{Element-wise Estimation Error of Generalized Fused Lasso}{Zhang and Chatterjee}
\firstpageno{1}

\begin{document}

\title{Element-wise Estimation Error of Generalized Fused Lasso}

\author{\name Teng Zhang  \thanks{TZ is supported by NSF Grant CNS-1818500}   \email teng.zhang@ucf.edu \\
       \addr Department of Mathematics\\
       University of Central Florida\\
       Orlando, FL 32816, USA       \AND \name Sabyasachi Chatterjee  \thanks{SC is supported by NSF Grant DMS-1916375}      \email sc1706@illinois.edu    \\
       \addr Department of Statistics\\
       University of Illinois at Urbana Champaign\\
       Champaign, IL 61820, USA
}

\editor{}

\maketitle

\begin{abstract}
The main result of this article is that we obtain an elementwise error bound for the Fused Lasso estimator for any general convex loss function $\rho$. 
We then focus on the special cases when either $\rho$ is the square loss function (for mean regression) or is the quantile loss function (for quantile regression) for which we derive new pointwise error bounds. Even though error bounds for the usual Fused Lasso estimator and its quantile version have been studied before; our bound appears to be new. This is because all previous works bound a global loss function like the sum of squared error, or a sum of huber losses in the case of quantile regression in~\cite{padilla2021risk}. Clearly, element wise bounds are stronger than global loss error bounds as it reveals how the loss behaves locally at each point. Our element wise error bound also has a clean and explicit dependence on the tuning parameter $\lambda$ which informs the user of a good choice of $\lambda.$ In addition, our bound is nonasymptotic with explicit constants and 
is able to recover almost all the known results for Fused Lasso (both mean and quantile regression) with additional improvements in some cases. 
\end{abstract}


\section{Introduction}\label{sec:fused}
Fused Lasso or (1D) Total Variation Denoising is arguably one of the most fundamental signal processing methods and is a basic tool for a vast array of data analysis problems. The literature studying the statistical accuracy of the Fused Lasso method is vast; see \cite{mammen1997}, \cite{Harchaoui2010}, \cite{dalalyan2017},\cite{Lin2016,Lin20162},\cite{Ortelli2018},~\cite{ortelli2021prediction},\\~\cite{Guntuboyina2017}, \cite{ortelli2021prediction}, \cite{padilla2021risk} to name a few. In this paper, we revisit this issue and take a closer look at the statistical error of the Fused Lasso estimator with a general convex loss function. Specifically, given a data vector $y \in \R^n$, we study the following estimator for a tuning parameter $\lambda \geq 0$,

\begin{equation}\label{eq:optimization}
\hat{\theta} =\argmin_{\theta \in \reals} \sum_{i=1}^n\rho(y_i - \theta_i)+\lambda\sum_{i=1}^{n-1}|\theta_i- \theta_{i+1}|,
\end{equation}
where $\rho: \R \rightarrow \R$ is a given convex function. Obviously, different choices of $\rho$ yield different estimators. The most popular choice of $\rho$ is the square function $\rho(x) = \frac{x^2}{2}$ which corresponds to the usual Fused Lasso or the 1D Total Variation Denoising estimator. In this paper, we will also be interested in the quantile regression version of Fused Lasso where $\rho$ is the quantile check function defined as 

\begin{equation}
\rho(x)=\begin{cases}\tau |x|,\,\,\text{if $x\geq 0$}\\ (1-\tau)|x|,\,\,\text{if $x< 0$}\end{cases}
\end{equation}
where $\tau$ is any fixed quantile level between $0$ and $1.$



\subsection{Informal Theorem and Proof Strategy}
Throughout the paper, we focus on the case when the true signal $\theta^*$ is piecewise constant with an unknown number of pieces and unknown changepoints. 
The main result of this paper is a general theorem giving pointwise error bounds for a certain transformation of $\hat{\theta}_i - \theta^*_i$ for any $1 \leq i \leq n$; see Theorem~\ref{thm:general}. To the best of our knowledge, such a theorem handling a general loss function $\rho$ and providing pointwise error bounds for Fused Lasso is new. 
We then apply this theorem to the particular cases of quantile and mean regression. For example, for the particular case of mean regression, our bound can be written informally as
\begin{theorem}[Informal]\label{thm:informal}
	Fix any $i$ such that $1 \leq i \leq n$. Let $d_i$ denote the distance of $i$ to the nearest change point of $\theta^*$ and $l_i$ denote the length of the constant piece of $\theta^*$ containing $i.$ With high probability, we have
	\begin{equation*}
	|\hat{\theta}_i - \theta^*_i| \lesssim \frac{1}{\sqrt{d_i}} + \frac{1}{\lambda} + \frac{\lambda}{l_i}.
 	\end{equation*}
 	Here, the $\lesssim$ notation hides certain $\log \log$ factors.
\end{theorem}

This bound reveals a few things:
 
\begin{enumerate}
	\item With a theoretical choice of $\lambda = \sqrt{l_i}$ one obtains a $\tilde{O}(\frac{1}{\sqrt{d_i}})$ rate of convergence for $|\hat{\theta}_i - \theta^*_i|.$ This means that within any level set of $\theta^*$, for most points in the interior, the rate of convergence is $\tilde{O}(\frac{1}{\sqrt{l_i}})$. Note that an oracle who knows the the level set of $\theta^*$ containing the index $i$ would just estimate by the mean within this block and would incur a rate of convergence exactly $O(\frac{1}{\sqrt{l_i}})$ as well.

  \item At the boundary of the level sets the rate of convergence can be bad. For instance, at the change points themselves, the above bound gives a $\tilde{O}(1)$ rate of convergence. This is validated by our simulations where we indeed find that the TVD estimator is inconsistent at the true change points.

  \item Consider the typical and realistic case where the lengths of all the $k$ pieces of $\theta^*$ are comparable up to a constant factor; equivalently each of the lengths are roughly $O(n/k)$. In this case, the optimal choice of $\lambda$ reduces to $\tilde{O}(\sqrt{n/k})$ which is what is recommended in the existing literature. Under this choice, squaring the bound we get $\tilde{O}(1/d_i)$ which when summed over all $i$ gives the mean squared error bound $\tilde{O}(k/n) \log n$ rate which is known to be unimprovable for any given $k \geq 1.$ This is how one can recover existing MSE bounds from our bound. 
  \end{enumerate}

For a general $\rho$ function and in particular quantile regression; we obtain a similar bound as in Theorem~\ref{thm:informal} except that the left hand side of our bound is a certain transformation of $|\hat{\theta}_i - \theta^*_i|.$ This transformation can be thought of as the natural loss function to measure error in our problem. If one desires, one can invert this transformation to get a bound on $\hat{\theta}_i - \theta^*_i$ which is what we have carried out for quantile regression; see Section~\ref{sec:quantile}.

The proof of Theorem~\ref{thm:general} relies on a few high level steps which we mention here for the convenience of the reader.  

\begin{enumerate}
	\item \textbf{Reduction to one piece}:  
	Fix any $i \in [1:n].$ Consider the constant piece of $\theta^*$; say $[e:l]$ containing $i$. The Fused Lasso solution $\hat{\theta}_i$ is a function of the entire data vector $y.$ However, \textit{Lemma~\ref{lemma:intermediate} shows that it is possible to bound $\hat{\theta}_i$ in terms of only the data $y_{e:l}$ in the block $[e:l].$} This is perhaps the most crucial observation facilitating our proof. This makes it possible to 
	analyze each constant piece of $\theta^*$ separately.  
	
	\item \textbf{Using subgradient information for a constant piece}: Reducing attention to the constant piece $[e:l]$ containing the index $i$ as in above, Lemma~\ref{lemma:intermediate} then uses appropriately chosen subgradient information at the optima $\hat{\theta}$ to derive a bound on $\hat{\theta}_i.$
	
	\item \textbf{Bound averages over intervals uniformly by a nonasymptotic law of iterated logarithm}: The bound given by Lemma~\ref{lemma:intermediate} can be expressed in terms of a maximum of a sum/average of i.i.d random variables over all possible intervals containing $i.$ We then use a nonasymptotic law of iterated logarithm to further bound this term. This is carried out in Proposition~\ref{prop:inter}. 
\end{enumerate}

\subsection{Related Work}
The literature on total variation penalized estimators and Fused Lasso is quite vast; e.g. see ~\cite{Rudin1992,Tibshirani2005,steidl2006splines,Tibshirani2008,rinaldo2009} and references therein. The estimation error of 1D total variation denoising or Fused Lasso has been well-studied in a variety of papers. These works analyze the convergence rate of the mean squared error $\frac{1}{n}\sum_{i=1}^n(\hat{\theta}_i-\theta_i^*)^2$ as $n\rightarrow\infty$. 

One of the earlier works in this area, Mammen and Geer \cite[Theorem 9]{mammen1997}, obtain rates of convergence for a least-squares estimator penalized by a continuous version of total variation of the $k$-th derivative of the regression function (locally adaptive regression splines). When $k=1$, this estimator reduces to the continuously penalized version of the Fused Lasso estimator. In particular, for $k = 1$, Mammen and Geer establish a rate of $O(n^{-2/3})$, which is the minimax rate in the space of functions of bounded total variation~\cite{10.1093/biomet/81.3.425}. This work turned out to be a precursor to the study of Trend Filtering estimators (discretely penalized version of locally adaptive regression splines),~\cite{kim2009ell_1}~\cite{tibshirani2014},~\cite{tibshirani2020divided} of general orders, of which the Fused Lasso is a special case.

There has been quite a lot of work recently on establishing a \textit{fast rate} (with a near parametric rate of convergence $O(K/n)$) for the usual mean regression version of Trend Filtering when the true underlying signal $\theta^*$ is assumed to be piecewise constant/polynomial (actually discrete spline) with $K$ change points, see~\cite{Harchaoui2010},~\cite{dalalyan2017},~\cite{Lin2016,Lin20162},~\cite{Ortelli2018},~\cite{ortelli2021prediction},\\~\cite{Guntuboyina2017}. The quantile regression version of Trend Filtering has been studied a lot less. Recently, \cite{padilla2021risk} has established a similar fast rate bound for quantile trend filtering.

All the above mentioned existing works analyze the asymptotic rate of  a global loss such as mean squared error. Secondly, these methods only work with a specific choice of $\rho$. There is no existing result which gives a unified treatment for any convex loss function $\rho$ and gives pointwise error bounds. In this paper we attempt to fill this gap in the literature by establishing a nonasymptotic upper bound of the elementwise loss $|\hat{\theta}_i-\theta_i^*|$, which is a much more precise characterization of the estimation errors. In addition, our bound can be applied to any choice of convex loss function $\rho$. Our elementwise bound also implies tight bounds for the global loss: our global loss bound on the quantile fused lasso estimator improves over the current best result and our global loss bound on the fused lasso estimator is slightly worse than the current best result by a factor of $\log n$. More detailed comparisons of our bounds with existing ones in the literature are given in Sections~\ref{sec:quantile},~\ref{sec:mean}.


There is another strand of results on Trend Filtering which establishes \textit{slow rates}, e.g the Fused Lasso attains the $O(n^{-2/3})$ rate for bounded variation functions.
For example, slow rate for a general $r$th order trend filtering estimator has been shown in \cite{tibshirani2014} which imply that it is  minimax rate optimal for the set of functions whose $r$-th derivative is of bounded variation. We do not focus on slow rates in this paper. We note that once fast rates are available, slow rates will naturally follow from an oracle risk bound known for Fused Lasso; e.g see Theorem $3.4$ in~\cite{ortelli2021prediction}. Similar oracle risk bounds should be proveable for Quantile Regression which we leave for future work.



%
%
%

\subsection{Outline}
The rest of this paper is structured as follows. In Section~\ref{sec:general}, we state our main result in Theorem~\ref{thm:general}. We then give a complete proof of Theorem~\ref{thm:general} in Section~\ref{sec:proof}. In Section~\ref{sec:quantile}, we focus on the Quantile Fused Lasso estimator and use Theorem~\ref{thm:general} to derive new and improved bounds. Similarly, in Section~\ref{sec:mean}, we focus on the usual Fused Lasso estimator and explore consequences of Theorem~\ref{thm:general}. Detailed comparisons of our bounds with existing ones in the literature are given in both Section~\ref{sec:quantile} and Section~\ref{sec:mean}.



\section{A General Bound}\label{sec:general}

\subsection{Notations}
Let $\rho: \R \rightarrow \R$ be a convex function. A standard fact about convex functions is that the left and right derivatives exist at all points. We denote the left and right derivative function of $\rho$ by $\rho'_{-}$ and $\rho'_{+}.$ Both $\rho'_{-}$ and $\rho'_{+}$ are non decreasing functions on $\R.$

We denote the set of positive integers at most $n$ by $[1:n]$ for any positive integer $n.$
For any vector $\theta \in \R^n$, using standard terminology, we term the corresponding penalty term in the objective function in~\eqref{eq:optimization} as its \textit{total variation} denoted by
\begin{equation*}
TV(\theta) = \sum_{i=1}^{n-1}|\theta_i- \theta_{i+1}|.
\end{equation*}

In this paper, we follow  \cite[Definition 2.2]{wainwright2019high}  and say that a random variable $X$ is sub-Gaussian with parameter $\sigma$, if $\Expect[e^{\lambda(X-\Expect[X])}]\leq e^{\sigma^2\lambda^2/2}$ for all $\lambda\in\mathbb{R}$.

\subsection{A General Elementwise Estimation Error Bound}
We now state a general elementwise error bound for the fused lasso estimator defined in~\eqref{eq:optimization}. Suppose the observations $\{y_i\}_{i=1}^n$ are generated from
\begin{equation}\label{eq:model}
y_i=\theta^*_i+\epsilon_i, 1\leq i\leq n,\,\,\,\, \theta^*_i, y_i,\epsilon_i\in\reals,
\end{equation}
where $\{\epsilon_i\}_{i=1}^n$ represents the noise and $\{\theta^*_i\}_{i=1}^n$ is a piecewise constant mean sequence with $K$ pieces. Then there exists changepoints $\{n_k\}_{k=1}^K$ such that $1=n_1<n_2<\cdots<n_K<n$ and for any $1\leq k\leq K$, $
\theta^*_{n_k}=\theta^*_{n_k+1}=\cdots=\theta^*_{n_{k+1}-1},$. Throughout the paper, we let $n_{K+1}=n+1$ for simplicity of notation.

\textbf{Assumption A}: The errors $\epsilon_i$ are i.i.d with common distribution $D$ such that $\rho'_{+}(\epsilon_i-t)$ and $\rho'_{-}(\epsilon_i-t)$ are sub-Gaussian random variables with parameter $\sigma$ for any $t\in\reals$.

\bigskip

We now introduce the \textit{natural} loss function in our problem. Define the functions $L^{+},L^{-}: \reals \rightarrow \reals^{+}$ as follows:
\begin{equation*}
\begin{cases}
L^{+}(t) = \Expect_{\epsilon \sim D} \rho'_{+}(\epsilon - t) \\
L^{-}(t) = \Expect_{\epsilon \sim D} \rho'_{-}(\epsilon - t).
\end{cases}
\end{equation*}

We are now ready to state our general element-wise estimation error bound.

\begin{thm}\label{thm:general}[General Elementwise Error Bound]

	Let $\hat{\theta}$ denote the fused lasso estimator defined in~\eqref{eq:optimization} when $y = \theta^* + \epsilon$ is the input data. Suppose \textbf{Assumption A} holds. Fix any $i \in [1:n].$  Then the following is true for any $\delta\in (0,\log 2/e)$:
	\begin{enumerate}
		\item
		$\Pr\Big(L^{+}(\hat{\theta}_i - \theta^*_i) \leq -B_{i,\delta} \Big)\leq \Big(1 + \frac{24}{(\log 2)^2}\Big)\delta^2$,
		\item
		$\Pr\Big(L^{-}(\hat{\theta}_i - \theta^*_i) \geq B_{i,\delta} \Big)\leq \Big(1 + \frac{24}{(\log 2)^2}\Big)\delta^2$,
	\end{enumerate}	
	where the number $B_{i,\delta} > 0$ is defined as 
	\begin{equation}\label{eq:boundterm}
	B_{i,\delta} = 4\sigma\!\!\left(\!\sqrt{\frac{\log\log 2\max(3,d_i)}{{\max(3,\!d_i)}}}\!+\!\sqrt{\frac{\log \!\frac{1}{\delta}}{d_i}}\!\!\right)\! + 4 \sigma^2\frac{\log\log(2m_{k(i)}\!) \!+\! \log \!\frac{1}{\delta}}{\lambda} + \frac{2\sqrt{m_{k(i)}\sigma^2\log \!\frac{1}{\delta}}\!+\!2\lambda}{m_{k(i)}}.
	\end{equation}
	Here, we define the positive integer $k(i) \in [1:K]$ such that  the $i$-th point lies in the $k(i)$-th constant piece of $\theta^*$, i..e,,  $n_{k(i)}\leq i\leq n_{k(i)+1}$, and $d_i=\min(i+1-n_{k(i)}, n_{k(i)+1}-i)$ is its distance to the nearest change point.	Moreover, $m_{k(i)}$ denotes the length of the constant piece of $\theta^*$ which contains $i.$
\end{thm}



We now make some remarks explaining and discussing the above theorem.
\begin{remark}
	As we will show subsequently, for both the mean regression and the quantile regression case, the $\rho$ function will satisfy the subgaussianity assumption in \textbf{Assumption A} on the random variables $\rho'_{+}(\epsilon - t)$ and $\rho'_{-}(\epsilon - t)$ under natural assumptions on the error random variable $\epsilon.$ In general, if $\rho$ is lipschitz (true for quantile regression) then $\rho'_{+},\rho'_{-}$ are bounded and thus automatically subgaussian without any tail decay assumptions on $\epsilon$ itself. In such cases, the above theorem holds for arbitrary heavy tailed $\epsilon$, a point discussed more in Section~\ref{sec:quantile}.
\end{remark}


\begin{remark}
	Theorem~\ref{thm:general} says that the functions $L^{+}$ and $L^{-}$ are the \textit{natural} loss functions (to measure the performance of $\hat{\theta}$) in our setup, in the sense that with these loss functions one can write such a unifying and general bound for any $\rho$ satisfying \textbf{Assumption A}. Due to the possible non symmetry of the convex function $\rho$ around $0$, we need both the loss functions $L^{+}$ and $L^{-}$ to capture the deviation of $\hat{\theta}$ from $\theta^*$ on both sides. 
	\end{remark}

\begin{remark}
	Operationally, the above theorem also gives an upper or lower bound on any particular element $\hat{\theta}_i - \theta^*_i$ as we now explain. The first assertion in Theorem $1$ says that with high probability the event $\{L^{+}(\hat{\theta}_i - \theta^*_i) \geq -B_{i,\delta}\}$ holds. Since $L^{+}$ is a monotone non increasing function and further, suppose it is strictly monotone in a neighborhood around $0$, then we can invert $L^{+}$ locally and thus we can rewrite the above event as $\{(\hat{\theta}_i - \theta^*_i) \geq (L^{+})^{-1}\left(-B_{i,\delta}\right)\}.$ Since $B_{i,\delta}$ will typically be a very small number and $(L^{+})^{-1}(0) = 0$ (true for both mean and quantile regression), then $(L^{+})^{-1}\left(-B_{i,\delta}\right)$ will also be a small number. Thus, by inverting $L^{+}$, we can get a lower bound on $\hat{\theta}_i - \theta^*_i.$ By a similar logic, by inverting $L^{+}$, one can obtain an upper bound on $\hat{\theta}_i - \theta^*_i.$


\end{remark}

\begin{remark}
	For ease of interpretation the reader can read $B_{\delta,i}$ as a sum of three terms $\frac{1}{\sqrt{d_i}} + \frac{1}{\lambda} + \frac{\lambda}{m_{k(i)}}.$ This reveals that the dependence of the error $\hat{\theta}_i - \theta^*_i$ on the local structure of $\theta^*$ around $\theta^*_i$ is explicit. In particular, this dependence is only through $d_i$, the distance of $i$ to the nearest change point of $\theta^*$ and $m_{k(i)}$, the length of the constant piece of $\theta^*$ containing $i.$  The dependence of the bound on the tuning parameter $\lambda$ is also clean and explicit.
	\end{remark}

\begin{remark}
	Another notable aspect of the bound given in the above theorem is that all the constants are explicit and small. This makes our bound truly nonasymptotic.
\end{remark}

\section{Proof of the General Bound}\label{sec:proof}
\subsection{An intermediate optimization problem and its elementwise estimation error}
We first present an important intermediate lemma that concerns the optimization problem \eqref{eq:optimization1}. Compared with the objective function in \eqref{eq:optimization}, the objective function in this optimization problem is similar with two additional terms $|\theta_1-a|$ and $|\theta_m-b|$ for any arbitrary real numbers $a,b.$ In particular, Lemma~\ref{lemma:intermediate} is completely deterministic and establishes a subset inclusion relation.

\begin{lemma}\label{lemma:intermediate}

Fix a data vector $y \in \R^m.$ For any two real numbers $a,b$, consider the solution to the $m$ dimensional optimization problem 
\begin{equation}\label{eq:optimization0.1}
\tilde{\theta}^{(a,b)} =\argmin_{\theta \in \R^m} G^{(a,b)}(\theta),
\end{equation}
where \[
 G^{(a,b)}(\theta)=\sum_{i=1}^m\rho(y_i-\theta_i)+\lambda \left(|\theta_1-a|+|\theta_m-b|+ TV(\theta)\right)\]
 and $\rho: \reals\rightarrow\reals$ is convex.  
 Then we have for any $i \in [1:m]$ and any $\alpha \geq 0$, 
 \begin{align}\label{eq:intermediate11}
\Big\{\sup_{a,b \in \R} \tilde{\theta}^{(a,b)}_i\geq \alpha \Big\}&\subseteq \Big\{\exists s,t: 1\leq s\leq i\leq t\leq m, z_1(s,t)\geq 0 \Big\}\\
\Big\{\inf_{a,b \in \R} \tilde{\theta}^{(a,b)}_i\leq -\alpha \Big\}&\subseteq \Big\{\exists s,t: 1\leq s\leq i\leq t\leq m, z_2(s,t)\leq 0 \Big\}\label{eq:intermediate2}.
 \end{align}
 where 
 \[
 z_1(s,t)= \begin{cases}\sum_{j=s}^t \rho'_{+}(y_j-\alpha)-2\lambda,\,\,&\text{if $s\neq 1$ and $t\neq m$}\\\sum_{j=s}^t \rho'_{+}(y_j-\alpha) ,\,\,&\text{if $s\neq 1, t=m$ or $s=1,t\neq m$}\\
\sum_{j=s}^t \rho'_{+}(y_j-\alpha) +2\lambda,\,\,&\text{if $s =1$ and $t= m$}\end{cases}\]
and
 \[
 z_2(s,t)= \begin{cases}\sum_{j=s}^t \rho'_{-}(y_j-\alpha)+2\lambda,\,\,&\text{if $s\neq 1$ and $t\neq m$}\\\sum_{j=s}^t \rho'_{-}(y_j-\alpha) ,\,\,&\text{if $s\neq 1, t=m$ or $s=1,t\neq m$}\\
\sum_{j=s}^t \rho'_{-}(y_j-\alpha) -2\lambda,\,\,&\text{if $s =1$ and $t= m$}\end{cases}\]
\end{lemma}

\begin{remark}
	A notable aspect of the above deterministic lemma is that the events or sets in the right hand side of~\eqref{eq:intermediate1} and~\eqref{eq:intermediate2} do not depend on $a,b.$ This fact plays a key role in our overall analysis. 
	\end{remark}

 \begin{proof}[Proof of Lemma~\ref{lemma:intermediate}]

Let us fix any $a,b \in \R$. We will drop the superscript $(a,b)$ notation within this proof to reduce notational clutter. 

Fix $i \in [1:m].$ First, we define $s,t$ as follows: $s\leq t$ are integers, and $[s,t]$ is the largest interval containing $i$ such that $\tilde{\theta}_j$ is greater or equal to $\tilde{\theta}_i$ for all $j\in [s,t]$. Note that this interval is always non empty as it always contains $i.$ This definition also implies that $\tilde{\theta}_{s-1} < \tilde{\theta}_i$ if $s \neq 1$, 
and $\tilde{\theta}_{t+1}<\tilde{\theta}_i$ if $t \neq m$. 

Secondly, for any $\eta > 0$ let us define a $\eta$ perturbation of $\tilde{\theta}$, denoted by $\tilde{\theta}'$ and defined as
\[
\tilde{\theta}'_j=\begin{cases}\tilde{\theta}_j-\eta,\,\,&\text{if $j\in [s,t]$,}\\\tilde{\theta}_j,\,\,&\text{if $j\not\in [s,t]$}.\end{cases}
\]

For all $j \in [s:t]$, 
\begin{equation}\label{eq:property_partial_rho2}
\lim_{\eta\rightarrow 0^+}\frac{\rho(y_j-\tilde{\theta}_j')-\rho(y_j-\tilde{\theta}_j)}{\eta}=\rho'_{+}(y_j-\tilde{\theta}_j).
\end{equation}

Let us denote the $m + 2$ dimensional vector $(a,\tilde{\theta}_1,\dots,\tilde{\theta}_m,b)$ by $[a:\tilde{\theta}:b].$ Similarly, we use the notation $[a:\tilde{\theta}^{'}:b]$ to denote the vector where we concatenate $a$ and $b$ to both ends of $\tilde{\theta}^{'}.$

Since $\tilde{\theta}_{s-1} < \tilde{\theta}_i$ if $s \neq 1$
and $\tilde{\theta}_{t+1}<\tilde{\theta}_i$ if $t \neq m$, for a small enough $\eta > 0$ one can check that the following holds:
\begin{equation*}
TV([a\!:\!\tilde{\theta}^{'}\!:\!b]) - TV([a\!:\!\tilde{\theta}\!:\!b]) \!= \!\begin{cases}
-2 \eta \lambda,&\text{if $s\neq 1$ and $t\neq m$}\\
-2 \eta \lambda + 2 \eta \lambda I(\tilde{\theta}_1<a),&\text{if $s=1,t\neq m$}\\
-2 \eta \lambda + 2\eta \lambda I(\tilde{\theta}_m<b),&\text{if $s\neq 1, t=m$ }\\
-2 \eta \lambda + 2\eta \lambda I(\tilde{\theta}_1<a)+ 2 \eta \lambda I(\tilde{\theta}_m<b),&\text{if $s =1$ and $t= m$}.\end{cases}
\end{equation*}


Therefore, using the last two displays and the fact that $\tilde{\theta}$ is the optimizer, we can conclude that 
\begin{align}\label{eq:subgradient}
&0 \leq \lim_{\eta\rightarrow 0+}\frac{G(\{\tilde{\theta}'_j\}_{j=1}^m)-G(\{\tilde{\theta}_j\}_{j=1}^m)}{\eta} \\ =&   \begin{cases}\sum_{j=s}^t\rho'_{+}(y_j -\tilde{\theta}_j ) -2\lambda,\,\,&\text{if $s\neq 1$ and $t\neq m$}\\\sum_{j=s}^t\rho'_{+}(y_j -\tilde{\theta}_j )-2\lambda + 2\lambda I(\tilde{\theta}_1<a),\,\,&\text{if $s=1,t\neq m$}\\\sum_{j=s}^t\rho'_{+}(y_j -\tilde{\theta}_j )-2\lambda + 2\lambda I(\tilde{\theta}_m<b),\,\,&\text{if $s\neq 1, t=m$ }\\
\sum_{j=s}^t\rho'_{+}(y_j -\tilde{\theta}_j )-2\lambda + 2\lambda I(\tilde{\theta}_1<a)+ 2\lambda I(\tilde{\theta}_m<b),\,\,&\text{if $s =1$ and $t= m$}.\end{cases}\nonumber
\end{align}


Finally, observe that when $\tilde{\theta}_i\geq \alpha$, the convexity of $\rho$ and $\tilde{\theta}_j\geq \tilde{\theta}_i\geq \alpha$ for $j \in [s:t]$ imply that $\rho'_{+}(y_i-\tilde{\theta}_j)\leq  \rho'_{+}(y_i-\alpha)$. The statement in~\eqref{eq:intermediate1} now follows. The proof of \eqref{eq:intermediate2} is similar.
 \end{proof}

We will now state the following lemma, which is a restatement of the finite version of law of iterated logarithm from \cite[Lemma 1]{pmlr-v35-jamieson14} (by setting $\epsilon=1$ in their notation):

\begin{lemma}\label{lem:lil}[Anytime Finite Sample Version of Law of Iterated Logarithm]

Let $X_1,X_2,\cdots$ be a sequence of i.i.d. sub-Gaussian random variables with mean $\mu$ and scale parameter $\sigma$.  For any $\delta\in (0,\log 2/e)$ we have
\begin{align*}
&\Pr\left(\bigcap_{t = 1}^{\infty} \left\{-4\sigma\sqrt{t\Big(\log\log(2t)+\log \frac{1}{\delta}\Big)}\leq \sum_{s=1}^tX_s-t\mu\leq  4\sigma\sqrt{t\Big(\log\log(2t)+\log \frac{1}{\delta}\Big)}\right\}\right) \\\geq& 1  \:\:\:-\:\: 6 \frac{\delta^2}{(\log 2)^2}.
\end{align*}
\end{lemma}


\medskip


\begin{proof}[Proof of Theorem~\ref{thm:general}]

The true signal $\theta^*$ has $K$ constant pieces which are intervals of the form  $[n_1,n_2-1]$, $[n_2, n_3-1]$, $\cdots$, $[n_K,n]$). In this proof we will analyze each piece separately. Let us focus on any one of the $K$ pieces; say the $k$th piece or in other words, the interval $I_k = [n_k, n_{k + 1} - 1].$ Without loss of generality, we can assume that $\theta^*_{I_k} = \textbf{0} \in \reals^{m_{k}}.$ This is because adding a constant to the data vector $y$ changes the solution $\hat{\theta}$ by the same constant.

Now we make the observation that $\hat{\theta}_{I_k}$ can be written as a solution of the following optimization problem:
\[
\{\hat{\theta}_i\}_{i \in I_k}=\argmin_{\{{\theta}_i\}_{i \in I_k}} \sum_{i \in I_k} \rho(y_i-\theta_i)+\lambda \left(|\theta_{n_k}-\hat{\theta}_{n_k-1}|+|\theta_{n_{k+1}}-\hat{\theta}_{n_{k+1}+1}|+ TV(\{\theta_i\}_{i=n_k}^{n_{k+1}-1})\right)
\] 

Here we are implicitly assuming that this piece is not the first or the last one. However, as the proof will reveal, the same conclusion can also be reached for the first and the last piece.

Now note that the above display is the same optimization problem as in \eqref{eq:optimization0.1} with $a=\hat{\theta}_{n_k-1}$ and $b=\hat{\theta}_{n_{k+1}+1}.$ Since Lemma~\ref{lemma:intermediate} gives a bound on the entries of the solution of such an optimization problem uniformly over $a,b$ we will now use Lemma~\ref{lemma:intermediate}. Infact, Lemma~\ref{lemma:intermediate} along with Lemma~\ref{lem:lil} actually implies the following result which we state as a proposition.

\begin{proposition}\label{prop:inter}
	For any vector $y \in \R^m$ and for any two real numbers $a,b$ consider the solution to the $m$ dimensional optimization problem
	\begin{equation}\label{eq:optimization1}
	\hat{\theta}^{(a,b)} =\argmin_{\theta \in \R^m} G^{(a,b)}(\theta),
	\end{equation}
	where \[
	G^{(a,b)}(\theta)=\sum_{i=1}^m\rho(y_i-\theta_i)+\lambda \left(|\theta_1-a|+|\theta_m-b|+ TV(\theta)\right)\]
	and $\rho: \reals\rightarrow\reals$ is convex.

	Suppose the input data vector $y = \theta^* + \epsilon$ is a random vector where the true signal vector $\theta^* = \textbf{0} \in \R^m$ and the noise vector $\epsilon$ satisfies \textbf{Assumption A}.

	For any $i \in [1:m]$ and any $\delta > 0$, define the number 
	$$B_{\delta,i} = 4\sigma \left(\sqrt{\frac{\log\log 2\max(3,\kappa_i)}{{\max(3,\kappa_i)}}}+\sqrt{\frac{\log \frac{1}{\delta}}{\kappa_i}}\right) + 4 \sigma^2\frac{\log\log(2m) + \log \frac{1}{\delta}}{\lambda} + \frac{2\sqrt{m\sigma^2\log \frac{1}{\delta}}+2\lambda}{m},$$where $\kappa_i=\min(i,m-i+1)$. 	
	
	For any $\alpha > 0$, if ${L^{+}(\alpha) }\leq - B_{\delta,i}$ then 
    \begin{equation}\label{eq:prop:iter1}
     \Pr(\sup_{a,b \in \R} \hat{\theta}^{(a,b)}_i > \alpha) \leq \Big(1 + \frac{24}{(\log 2)^2}\Big) \delta^2.\end{equation}
    Similarly, if ${L^{-}(\alpha)} \geq B_{\delta,i}$ then 
    \begin{equation}\label{eq:prop:iter2}
     \Pr(\inf_{a,b \in \R} \hat{\theta}^{(a,b)}_i < -\alpha) \leq \Big(1 + \frac{24}{(\log 2)^2}\Big) \delta^2.\end{equation}

\end{proposition}

We now claim that the above proposition can be directly applied to $\hat{\theta}_{I_k}$ to obtain the desired bounds in Theorem~\ref{thm:general} by plugging in $m = m_{k(i)}$ and $\kappa_i = d_{i}$ in the expression for $B_{\delta_i}.$

For example, to show that the first assertion in Theorem~\ref{thm:general} follows from \eqref{eq:prop:iter1}, we need to ensure that there exists an $\alpha > 0$ such that $L^{+}(\alpha) \leq - B_{\delta,i}.$ This is because if such an $\alpha > 0$ exists then for this $\alpha$ we have the event inclusion relation (since $L^{+}$ is non increasing),
$$\{L^{+}(\sup_{a,b \in \R} \hat{\theta}^{(a,b)}_i) < L^{+}(\alpha) \leq - B_{\delta,i}\} \subset \{\sup_{a,b \in \R} \hat{\theta}^{(a,b)}_i > \alpha\}.$$
The above display along with~\eqref{eq:prop:iter1} implies that 
$$\Pr\Big(L^{+}(\hat{\theta}_i - \theta^*_i) \leq -B_{i,\delta} \Big)\leq\Big(1 + \frac{24}{(\log 2)^2}\Big)\delta^2$$
which is the same as the first assertion in Theorem~\ref{thm:general}.

Now note that the existence of an $\alpha > 0$ such that $L^{+}(\alpha) \leq - B_{\delta,i}.$ is automatically guaranteed if $-B_{\delta,i} \geq \inf_{x \in \R} L^{+}(x).$ In the complementary case when $-B_{\delta,i} < \inf_{x \in \R} L^{+}(x)$ then $\Pr\Big(L^{+}(\hat{\theta}_i - \theta^*_i) \leq -B_{i,\delta} \Big) = 0$ and hence the desired bound is anyway true. The second assertion can be proved similarly.

\end{proof}


	
	
Finally it remains to prove Proposition~\ref{prop:inter}.

\begin{proof}[Proof of Proposition~\ref{prop:inter}]
We will prove~\eqref{eq:prop:iter1} and~\eqref{eq:prop:iter2} can then be shown similarly. Within this proof, we will fix any arbitrary $a,b \in \R$ and show the final bound which is free from $a,b.$ We will also drop the superscript notation and simply denote $\hat{\theta}^{(a,b)}$ by $\hat{\theta}.$ 

By our assumption, $y_i = \epsilon_i.$ Within this proof, let us denote the partial sums of the random variables $\rho'_{+}(y_i-\alpha)$ within any interval $[s,t] \subset [1:m]$ by
$$z(s,t) = \sum_{i=s}^t \rho'_{+}(y_i-\alpha).$$

By the set inclusion relation in \eqref{eq:intermediate11} in  Lemma~\ref{lemma:intermediate}, we can now write for any $i \in [1:m]$ and any $\alpha > 0$,
\begin{equation}\label{eq:p1234}
\Pr\Big(\hat{\theta}_i > \alpha \Big)\leq p_1+p_2+p_3+p_4,
 \end{equation}
 where $p_1,p_2,p_3,p_4$ are probabilities of events given by 
 \begin{align*}
 p_1=&\Pr\Big(\exists s\neq 1, t\neq m: s\leq i\leq t,\,\,  z(s,t)\geq 2\lambda\Big)\\
 p_2=&\Pr\Big(\exists t\geq i: z(1,t)\geq 0\Big)\\
p_3=&\Pr\Big(\exists s\leq i: z(s,m)\geq 0\Big)\\
p_4=&\Pr\Big(z(1,m)\geq -2\lambda \Big).
\end{align*} 
We will now use Lemma~\ref{lem:lil}  to bound these probabilities. Fix $\delta > 0.$ Let us also denote $\mu = \Expect \rho'_{+}(y_1-\alpha) = L^{+}(\alpha).$


First, let us consider $p_3.$ Let us define the event 
$$E^{(1)}_{\delta} = \bigcap_{s = 1}^{i} \left\{z(s,m) \in [(m - s + 1) \mu \pm 4\sigma\sqrt{(m - s + 1) \Big(\log\log(2(m - s + 1))+\log \frac{1}{\delta}\Big)}]\right\}.$$
By Lemma~\ref{lem:lil} we have 
$$\Pr(E^{(1)}_{\delta}) \geq 1 - 6 \frac{\delta^2}{(\log 2)^2}.$$

It is not hard to check that if 
\begin{equation}\label{eq:p1}
\mu \leq -T_1, \,\,\text{where}\,\,T_1=4\sigma\max_{1\leq s\leq i}\sqrt{\frac{\log\log(2(m-s+1))+\log \frac{1}{\delta}}{m-s+1}}
\end{equation}
then the event $E^{(1)}_{\delta} \cap \{\exists s\leq i: z(s,m)\geq 0\}$ cannot happen. Therefore, if $\mu \leq -T_1$ then
$$p_3 = \Pr\Big(\exists s\leq i: z(s,m)\leq 0\Big) \leq \Pr\Big(\{\exists s\leq i: z(s,m)\leq 0\} \cap (E^{(1)}_{\delta})^c \Big)\leq 6 \frac{\delta^2}{(\log 2)^2}.$$

Similarly, we can conclude that $p_2\leq 6(\delta/\log 2)^2$ if 
\begin{equation}\label{eq:p2}
\mu \leq  -T_2, \,\,\text{where}\,\,T_2=4\sigma\max_{i\leq t\leq m}\sqrt{\frac{\log\log(2t) + \log \frac{1}{\delta}}{t}}.
\end{equation}

We will now bound $p_1.$ Define the events
\begin{equation*}
\begin{cases}
A = \{\exists t: i \leq t \leq m, z(i,t) \geq  \lambda\} \\
B =  \{\exists s: 1 \leq s < i, z(s,i - 1) \geq  \lambda\}.
\end{cases}
\end{equation*}

Now we can use a union bound argument and note that
\begin{equation*}
p_1 \leq \Pr(A) + \Pr(B).
\end{equation*}
We will now bound $\Pr(A)$ and $\Pr(B)$ can then be bounded similarly.

Let us define the event 
$$E_{\delta} = \bigcap_{t = i}^{m} \left\{z(i,t) \in [(t - i + 1) \mu \pm 4\sigma\sqrt{t\Big(\log\log(2(t - i + 1))+\log \frac{1}{\delta}\Big)}]\right\}.$$

Lemma~\ref{lem:lil} says that 
\begin{equation*}
\Pr(E_{\delta}) \geq 1 - 6 \frac{\delta^2}{(\log 2)^2}.
\end{equation*}

This implies in particular that if 
\begin{equation*}
\mu \leq -\max_{t \in [i:m]} \Big(4\sigma\sqrt{\frac{\log\log(2(t - i + 1))-\log\delta}{(t - i + 1)}} - \frac{\lambda}{(t - i + 1)}\Big)
\end{equation*}

then the event $A \cap E_{\delta}$ cannot happen.

Similarly, if \begin{equation*}
\mu \leq -\max_{s \in [1:(i - 1)]} \Big(4\sigma\sqrt{\frac{\log\log(2s)-\log\delta}{s}} - \frac{\lambda}{s}\Big)
\end{equation*}

then the event $B \cap E_{\delta}$ cannot happen.

Combining the last two displays lets us conclude that if 
$\mu \leq -T_3$, where $$T_3 = \max_{s \in [1:m]} \Big(4\sigma\sqrt{\frac{\log\log(2s)-\log\delta}{s}} - \frac{\lambda}{s}\Big)$$

then 
\begin{equation}\label{eq:p3}
p_1 \leq \Pr(A) + \Pr(B) = \Pr(A \cap E_{\delta}^c) + \Pr(B \cap E_{\delta}^c) \leq 2 \Pr(E_{\delta}^c) \leq 12 \frac{\delta^2}{(\log 2)^2}.
\end{equation}

Finally, an application of Hoeffding's inequality for sum of independent subgaussian random variables~\cite[Proposition 2.5]{wainwright2019high} implies that 
$$ p_4\leq \exp\Big(-\frac{({2\lambda}-m\mu)^2}{2m\sigma^2}\Big).$$

Equivalently, $p_4\leq \delta^2$ if 
\begin{equation}\label{eq:p4}
\mu\leq -T_4, \,\,\text{where}\,\,T_4 = \frac{2\sqrt{m\sigma^2\log \frac{1}{\delta}}+2\lambda}{m}.
\end{equation}
Combining~\eqref{eq:p1234} with the derived bounds on $p_1, p_2, p_3, p_4$ in \eqref{eq:p1}-\eqref{eq:p4}, we have 
$$ \Pr\Big(\hat{\theta}_i>\alpha \Big)\leq (\delta^2 +\frac{24}{(\log 2)^2}\delta^2)$$ if
\begin{align}\label{eq:T1234}
&-\mu \geq  \max(T_1,T_2,T_3,T_4).
\end{align}

At this point, we further bound
 \begin{equation}T_1 \!\leq\! 4\sigma \left(\!\max_{1\leq s\leq i}\!\sqrt{\frac{\log\log(2(m\!-\!s\!+\!1))}{m-s+1}}\!+\!\sqrt{\frac{\log \frac{1}{\delta}}{m-i+1}}\right) \!\!\!\leq 4\sigma \!\left(\!\sqrt{\frac{\log\log 2\max(3,\kappa_i)}{{\max(3,\kappa_i)}}}\!+\!\sqrt{\frac{\log \!\frac{1}{\delta}}{\kappa_i}}\!\right),\label{eq:T1_upperbound}\end{equation}
where $\kappa_i=\min(i,m-i+1)$.

The above inequality uses the fact that for the function $f(x)=\frac{\log\log 2x}{x}$, $f(1)\leq f(2)\leq f(3)$, and f(x) is nonincreasing for $x\geq 3$, since $f'(x)=\frac{\frac{1}{\log(2x)}-\log\log(2x)}{x^2}$ and $\frac{1}{\log(2x)}-\log\log(2x)<0$ for $x\geq 3$. 

A similar reasoning gives the same upper bound to $T_2$: 
\[T_2 \leq 4\sigma\left(\sqrt{\frac{\log\log 2\max(3,\kappa_i)}{{\max(3,\kappa_i)}}}+\sqrt{\frac{\log \frac{1}{\delta}}{\kappa_i}}\right),\]

Finally, we can bound
\[T_3\leq \max_{1 \leq s \leq m} \Big(4\sigma\sqrt{\frac{\log\log(2m) + \log \frac{1}{\delta}}{s}} - \frac{\lambda}{s}\Big)\leq 4 \sigma^2\frac{\log\log(2m) + \log \frac{1}{\delta}}{\lambda},\]
where we used the fact that 
$\max_{x \geq 1} \left(\frac{u}{\sqrt{x}} - \frac{v}{x}\right) \leq \frac{u^2}{4v}$ for any positive numbers $u,v.$

We can now define $B_{i,\delta}$ to be the sum of these bounds on $T_1,T_2,T_3$ along with $T_{4}$ itself. This finishes the proof. 
\end{proof}


\section{Quantile Fused Lasso Regression}\label{sec:quantile}
In the quantile regression set up, the data vector $y$ follows the model \eqref{eq:model} where the noise $\{\epsilon_i\}_{i=1}^n$ variables are i.i.d. sampled from a distribution $D$ with CDF $F$ whose $\tau-$th quantile is uniquely $0$. Then the piecewise constant sequence $\theta^*$ becomes a unique $\tau$ quantile sequence of the data $y.$ 
We record this formally as an assumption on the error variables.

\textbf{Assumption Q1}:
The error variables $\epsilon_i$ are i.i.d with distribution D and CDF $F$ such that 
$F(0) = \tau$ and $F$ is continuously strictly increasing at $\tau.$

\medskip

We want to study the Fused Lasso estimator defined in~\eqref{eq:optimization} with $\rho(x)$ chosen to be the $\tau$-th quantile loss function:
\begin{equation}\label{eq:qtlfn}
\rho(x)=\begin{cases}\tau |x|,\,\,\text{if $x\geq 0$}\\ (1-\tau)|x|,\,\,\text{if $x< 0$},\end{cases}.
\end{equation}

We note that for this choice of $\rho$, we have $\rho'_+(x)=\tau I(x\geq 0) -(1-\tau)I(x<0)$ and $\rho'_-(x)=\tau I(x> 0) -(1-\tau)I(x\leq0)$. Since $\rho'_+$ and $\rho'_-$ are bounded within an interval of length $1$, $\rho'_{+}(\epsilon_i-t)$ and $\rho'_{-}(\epsilon_i-t)$ are automatically sub-Gaussian random variables with parameter $1/2$ for any $t\in\reals$~\cite[Example 2.4, Exercise 2.4]{wainwright2019high}.  Therefore, \textbf{Assumption A}, required for Theorem~\ref{thm:general} to hold, is automatically satisfied here without any further assumptions on the error distribution $D.$


Before proceeding further, we characterize the loss functions $L^{+},L^{-}$ in this setup. For the CDF $F$, we denote its left limit function by $F^{-}$; 
$$F^{-}(t) = \lim_{s \uparrow t} F(s).$$

\begin{lemma}\label{lemma:lossfunction}
	Under \textbf{Assumption Q1}, for any $t > 0$ we have,
	\begin{equation*}
	L^{+}(t) = \Expect_{\epsilon \sim D} \rho'_+(\epsilon - t) = -\Pr(0 \leq \epsilon < t) = F^{-}(0) - F^{-}(t).
	\end{equation*}
	Similarly, for any $t < 0$ we have,
	\begin{equation*}
	L^{-}(t) = \Expect_{\epsilon \sim D} \rho'_-(\epsilon - t) = -\Pr(t < \epsilon \leq 0) = F(t) - F(0).
	\end{equation*}
	\end{lemma}
\begin{proof}[Proof of Lemma~\ref{lemma:lossfunction}] By assumption \textbf{Q1},  $0$ is the unique $\tau$-th quantile and we have \[\Expect_{\epsilon\sim D} \rho'_{+}(\epsilon)=\tau\Pr(\epsilon\geq 0)-(1-\tau)\Pr(\epsilon<0)=0.\] For $t>0$, 
\begin{align*}
&\Expect_{\epsilon\sim D} \rho'_{+}(\epsilon-t)=
\Expect_{\epsilon\sim D} \rho'_{+}(\epsilon-t)-\Expect_{\epsilon\sim D} \rho'_{+}(\epsilon)\\
=&\Big(\tau\Pr(\epsilon\geq t)-(1-\tau)\Pr(\epsilon<t)\Big)-\Big(\tau\Pr(\epsilon\geq 0)-(1-\tau)\Pr(\epsilon<0)\Big)\\
=&-\tau\Pr(0 \leq \epsilon < t)-(1-\tau)\Pr(0 \leq \epsilon < t)=-\Pr(0 \leq \epsilon < t).
\end{align*}
The proof for the case $t<0$ is similar. 
\end{proof}

A direct application of Theorem~\ref{thm:general} gives the following pointwise bound for the quantile regression version of Fused Lasso.

\begin{theorem}[Assumptionless Bound for Quantile Regression]\label{thm:quantile}
	
	Fix $0 < \tau < 1.$ Let $\hat{\theta}$ denote the fused lasso estimator defined in~\eqref{eq:optimization} when $y = \theta^* + \epsilon$ is the input data and $\rho$ is the function given in~\eqref{eq:qtlfn}. Suppose \textbf{Assumption Q1} holds. Fix any $i \in [1:n].$  Then the following is true for any $\delta\in (0,\log 2/e)$:
	\begin{enumerate}
		\item
		$\Pr\Big((F^{-}(\hat{\theta}_i - \theta^*_i) - F^{-}(0) \geq B_{i,\delta}  \Big)\leq \Big(1 + \frac{24}{(\log 2)^2}\Big) \delta^2$,
		\item
		$\Pr\Big(F(0) - F(\hat{\theta}_i - \theta^*_i)) \geq B_{i,\delta} \Big)\leq \Big(1 + \frac{24}{(\log 2)^2}\Big) \delta^2$,
	\end{enumerate}	
	where the number $B_{i,\delta}^{\text{quantile}} > 0$ is defined as 
	\begin{equation*}
		B_{i,\delta}^{\text{quantile}} \!= \!2\left(\!\sqrt{\frac{\log\log 2\max(3,d_i)}{{\max(3,d_i)}}}\!+\!\sqrt{\frac{\log \frac{1}{\delta}}{d_i}}\right) +  \frac{\log\log(2m_{k(i)}) + \log \frac{1}{\delta}}{\lambda} + \frac{\sqrt{m_{k(i)}\log \frac{1}{\delta}}\!+\!2\lambda}{m_{k(i)}}.
	\end{equation*}
\end{theorem}
	\begin{proof}
The result follows from applying Theorem~\ref{thm:general} with $\sigma=1/2$.
\end{proof}

\begin{remark}
	As discussed in Section~\ref{sec:discuss1}, the above result appears to be new for Quantile Fused Lasso. The reason we state such a result in terms of the CDF of the error distribution is that it allows us to present a clean result which holds for all entries $1 \in [1:n]$ with an explicit bound involving the tuning parameter $\lambda$ and the signal parameters under minimal assumptions.
	\end{remark}

\begin{remark}
	We stress on the fact that only \textbf{Assumption Q1} is needed on the distribution of the error variables for the above theorem to hold. This assumption ensures that $0$ is a unique $\tau$ th quantile of the errors and consequently, $\theta^*$ is a unique $\tau$ th quantile sequence of the data vector $y.$ Clearly, such an assumption is necessary to make sense of estimating the quantiles. So, the above theorem holds for all distributions satisfying \textbf{Assumption Q1}, including arbitrarily heavy tailed errors.
	\end{remark}


\subsection{Bounds on $\hat{\theta}_ i - \theta^*_i$}
Theorem~\ref{thm:quantile} gives a pointwise bound in terms of the CDF $F$ of the distribution of the errors. Under a slightly stronger assumption than \textbf{Assumption Q1} (stated below), it is possible to translate the above bound to a bound on the pointwise error $\hat{\theta}_i - \theta^*_i.$ The results here are necessarily slightly messier than in Theorem~\ref{thm:quantile} because of the need to invert $F.$

\textbf{Assumption Q2}:
The error variables $\epsilon_i$ are i.i.d with distribution D and CDF $F$ such that 
$F(0) = \tau$, and there exists a constant ${L} > 0$ such that for any $x \in [-1,1]$,
\begin{equation}\label{eq:q2}
|F(x) - F(0)| \geq {L} |x|.
\end{equation}

\bigskip

\begin{remark}
\textbf{Assumption Q2} is stronger than in the sense that it implies \textbf{Assumption Q1}. This assumption is a local linear growth assumption on the true cdf function $F$ in a neighborhood of $0.$ This assumption is pretty mild, since any distribution which has density (w.r.t to Lebesgue measure) which is bounded away from $0$ on the compact interval $[-1,1]$, will automatically satisfy this assumption. Therefore, this assumption does not prevent the error to have very heavy tails, like the Cauchy distribution. This type of an assumption is commonly made in the quantile regression literature; see Assumption $1$ in~\cite{padilla2021risk}, Condition $2$ in~\cite{he1994convergence} and Condition D.1 in~\cite{belloni2011l1}.
\end{remark}

Under this assumption, we may establish the following result on the elementwise error bound for fused quantile lasso, which follows from Theorem~\ref{thm:quantile}. We defer the proof to the appendix.
\begin{cor}[Elementwise error bound for quantile regression, under Assumption Q2]\label{cor:quantile2}
Suppose that Assumption Q2 holds. Fix any $i \in [1:n].$ If \begin{align}\label{eq:quantile2a_assumption}{B_{i,\delta}^{\text{quantile}}}\leq{{L}},\end{align} then the following is true for any $\delta\in (0,\log 2/e)$:
\begin{equation}\label{eq:quantile2a}
\\Pr\left(|\hat{\theta}_i -\theta_i^*|\leq \frac{B_{i,\delta}^{\text{quantile}}}{{L}} \right)\geq 1-2\Big(1 + \frac{24}{(\log 2)^2}\Big)\delta^2.
\end{equation}
\end{cor}

Corollary~\ref{cor:quantile2} is an elementwise bound for the elementwise error $|\hat{\theta}_i-\theta_i^*|$ which holds for all entries $i \in [1:n]$ satisfying assumption \eqref{eq:quantile2a_assumption}. The next lemma~\ref{lemma:assumption} gives a sufficient condition which shows that this assumption holds for most indices, and as a result, \eqref{eq:quantile2a} holds for most indices as well. The proof of this lemma is deferred to the Appendix.

\begin{lemma}\label{lemma:assumption} Let $m_{\min}=\min_{k=1,\cdots,K}m_k
$ represents the length of the shortest level set of $\theta^*.$ Then, if
\begin{equation}\label{eq:assumption_mmin}
m_{\min}\geq\frac{36}{L^2}\log\frac{1}{\delta}, \,\,
\frac{6(\log\log (2m_{\min})+\log\frac{1}{\delta})}{L}\leq\lambda\leq \frac{L}{12}m_{\min},
\end{equation}
then the assumption \eqref{eq:quantile2a_assumption} holds for all $i$ such that
\begin{equation}\label{eq:assumption_index}
d_i\geq \max\left(3,\frac{12^4}{L^4}, \frac{12^2}{L^2}\log\frac{1}{\delta}\right),
\end{equation}
i.e., the locations that are separated from the change points by a certain distance.
\end{lemma}
\begin{remark}
Lemma~\ref{lemma:assumption} implies that \eqref{eq:quantile2a_assumption} holds for at least $n-2K\max\left(3,\frac{12^4}{L^4}, \frac{12^2}{L^2}\log\frac{1}{\delta}\right)$ indices. If $L=O(1)$, $\delta=O(1)$, $K=o(n)$, then \eqref{eq:quantile2a_assumption} holds for almost all $1\leq i\leq n$ as $n\rightarrow\infty$. 
\end{remark}
%
Next, we present a bound on the sum of squared errors in Corollary~\ref{cor:estimationerror2}. 
The proof is based on the fact that the elementwise bound in \eqref{eq:quantile2a} holds for most indices as discussed in Lemma~\ref{lemma:assumption}, and a separate argument (see Lemma~\ref{lemma:uniformbound} and its proof) that establishes a crude uniform bound on the estimation errors over all $1\leq i\leq n$ of the following form
$$\min_{1\leq i\leq n}\theta^*_i-1\leq \min_{1\leq i\leq n}\hat{\theta}_i\leq \max_{1\leq i\leq n}\hat{\theta}_i\leq \max_{1\leq i\leq n}\theta^*_i+1.$$ 
The point is that, we only need to use the above crude bound on a few indices where \eqref{eq:quantile2a} does not hold. The detailed proof and algebraic calculations are deferred to the Appendix.

\begin{cor}[Sum of squared estimation errors for quantile regression]\label{cor:estimationerror2}
Under \textbf{Assumption Q2}, 
assuming that
\[
m_{\min}\geq\frac{18}{L^2}\log\frac{n}{\delta}, \,\, \frac{3(2\log\log (2n)+\log\frac{n}{\delta})}{L}\leq\lambda\leq \frac{L}{12}m_{\min},
\]
then the sum of squared errors is bounded above with high probability for all $\delta\in (0,(\log2/e)^2)$,
\begin{align}\nonumber
&\Pr\Bigg(\sum_{i=1}^n|\hat{\theta}_i -\theta_i^*|^2\leq \frac{24}{L^2}\Big(2\log\log 2n+\log\frac{n}{\delta}\Big)\Big(K+\sum_{k=1}^K\log\frac{m_k}{2}\Big)+ \frac{3n}{L^2}\Big(4\frac{\log\log^2(2n) + \log ^2\frac{n}{\delta}}{\lambda^2}\Big)\\&+ \frac{6K}{L^2}\log\frac{n}{\delta}+\frac{24\lambda^2}{L^2}\sum_{k=1}^K\frac{1}{m_k}
+2K\max\Big(3,\frac{12^4}{L^4}, \frac{12^2}{2L^2}\log\frac{n}{\delta}\Big)V^2\Bigg)\geq 1-4\Big(1 + \frac{24}{(\log 2)^2}\Big)\delta\label{eq:estimationerror2},
\end{align}  
where $V=\max_{1\leq i\leq n}\theta^*_i-\min_{1\leq i\leq n}\theta^*_i$.
\end{cor}

\begin{remark}
The above bound consists of several terms and maybe hard to read at a first glance. It is instructive to consider the special case, when $L=O(1)$, $\delta=o(1)$, and the length of all intervals are in the same order, i.e.,  $m_1=\cdots=m_K=O(n/K)$ which will hold for any realistic $\theta^*$ and noise distribution. In this case, the estimation error becomes  \eqref{eq:estimationerror2} becomes
\begin{equation}\label{eq:estimationerror3}
\sum_{i=1}^n|\hat{\theta}_i -\theta_i^*|^2\leq O\left(K \log n\log\frac{n}{K}+n\frac{\log^2n}{\lambda^2}+\frac{\lambda^2K^2}{n}+K\log nV^2\right). 
\end{equation}

Furthermore, if we set $\lambda=\log n\sqrt{n/K}$, our bound in~\eqref{eq:estimationerror3} becomes
\begin{align}\label{eq:estimationerror4}
\sum_{i=1}^n|\hat{\theta}_i-\theta_i^*|^2\leq O\left(K \log n(\log\frac{n}{K}+V^2)\right).
\end{align}

\end{remark}

\begin{remark}
	We have made a conscious effort to keep explicit constants in the bound~\eqref{eq:estimationerror2}. This is to highlight the fact that our proof technique yields truly nonasymptotic bounds with explicit dependence on the tuning parameter $\lambda$ and $V.$ It may be possible to obtain better constants but that is too delicate an issue and is beyond the scope of this paper. 
	\end{remark}

\subsection{Comparison with existing works on the estimation errors}\label{sec:discuss1}


As far as we are aware, the only result (before this work) giving error bounds for Quantile Fused Lasso appear in Padilla and Chatterjee \cite{padilla2021risk} who investigate the risk under \textbf{Assumption A} and also assuming that the total variation of $\{\theta^*_i\}_{i=1}^n$ is $O(1)$, which implies that $V=O(1)$ in our result. In particular, \cite[Theorem 4]{padilla2021risk} shows that for a particular choice of $\lambda$ defined up to an unspecified constant, see~\eqref{eq:lamdefn}, the risk bound measured in Huber loss is bounded by
\begin{equation}\label{eq:comparison1}
\sum_{i=1}^n\min (|\hat{\theta}_i-\theta_i^*|, |\hat{\theta}_i-\theta_i^*|^2)\leq C K\log(\frac{en}{K})\log n\log K
\end{equation}
with high probability.

We now give some points of comparison.
\begin{itemize}
\item The  LHS of our bound \eqref{eq:estimationerror4} is stronger in the sense that it bounds the sum of squared errors instead of the Huber loss as in~\eqref{eq:comparison1}. 

\item The RHS of our bound~\eqref{eq:estimationerror4} is smaller than the RHS of \eqref{eq:comparison1} by a factor of $\log K$.

\item Our bound depends explicitly on $\lambda$ for a large range of $\lambda$, thereby informing the user about a good choice of $\lambda.$ In contrast, their result only holds for a particular choice of $\lambda$ given below for some unspecified constant $c$
\begin{equation}\label{eq:lamdefn}
\lambda=c\max\left(\frac{K\log n\log K\log\frac{n}{K}}{V^*},\sqrt{\frac{n}{K}\log n}\right).
\end{equation}
Here $V^*$ represents the  total variation of the sequence $\{\theta^*_i\}_{i=1}^n$ and is assumed to be $O(1)$ (which is also stronger than our assumption that $V=O(1)$).

\item The dependence on $V$ is also explicit in our bound and our bound holds for any $V.$ In fact, it shows that the choice of $\lambda$ need not depend on $V$ as the term involving $V$ in~\eqref{eq:estimationerror2} does not involve $\lambda.$ In contrast, the bound in~\eqref{eq:comparison1} only holds when $V^* = O(1)$ and their choice of $\lambda$ needs to depend on $V^*.$

\item The bound in \eqref{eq:estimationerror2} is truly nonasymptotic, while  \cite{padilla2021risk} establish an asymptotic bound.
\end{itemize}

\noindent\textbf{An improved bound on the sum of squared errors}

Actually, we can prove a slightly stronger version of \eqref{eq:estimationerror2} where the term $\frac{24\lambda^2}{L^2}\sum_{k=1}^K \frac{1}{m_k}$ 
 is replaced with \[\frac{144\lambda^2}{L^2}\sum_{\eta_{k-1}\neq \eta_k}\frac{1}{m_k},\] where $\eta_k=\sign(\theta^*_{n_{k}}-\theta^*_{n_{k-1}})$ is the direction of the jump of the piecewise constant mean seuence from the $k$-the segment to the $k+1$-th segment, and we let $\eta_0=\eta_{K}=0$.
 
The proof of this fact is deferred to Section~\ref{sec:improved2} in the Appendix. This fact improves our result in the following way. Note that a minimum length condition ensuring that each $m_i = O(n/k)$ is needed for~\eqref{eq:estimationerror3} to hold.
However, this fact allows~\eqref{eq:estimationerror3} to hold under a less stringent minimum length assumption where we only need $m_i = O(n/k)$ for $i$ such that $\eta_i \neq \eta_{i - 1}$. For example, if $\theta^*$ is monotone we only need $m_1$ and $m_{K}$ to have length at least $O(n/K).$ This fact that one only requires a minimum length condition on these local optima (of $\theta^*$) blocks (plus the first and the last block) is known from before; see~\cite{Guntuboyina2017},~\cite{Ortelli2018}. Therefore, our result is in accordance with this fact.





\section{Mean Fused Lasso Regression}\label{sec:mean}
Our general theorem yields pointwise error bounds which are new even for the well studied usual Fused Lasso estimator which performs mean regression. In the mean regression set up, the data vector $y$ follows the model \eqref{eq:model} where the noise $\{\epsilon_i\}_{i=1}^n$ variables are i.i.d. sampled from a sub-Gaussian distribution with parameter $\sigma$ and zero mean. Then the piecewise constant sequence $\theta^*$ becomes a unique mean sequence of the data $y.$ 
We record this formally as an assumption on the error variables.

\textbf{Assumption Q3}:
The error variables $\epsilon_i$ are i.i.d with distribution D such that it is sub-Gaussian distribution with parameter $\sigma$ and has zero mean.

We want to study the Fused Lasso estimator defined in~\eqref{eq:optimization} with $\rho(x)=x^2/2$. For this choice of $\rho$, $\rho'_+(x)=\rho'_-(x)=x$, and $L^+(t)=L^-(t)=-t$.  Then a direct application of Theorem~\ref{thm:general} gives the following pointwise bound for the mean regression version of Fused Lasso.

\begin{thm}[Elementwise error bound for mean regression]\label{thm:mean_elementwise}
Let $\hat{\theta}$ denote the fused lasso estimator defined in~\eqref{eq:optimization} when $y = \theta^* + \epsilon$ is the input data and $\rho(x)=x^2/2$. Suppose \textbf{Assumption Q3} holds. Fix any $i \in [1:n].$  Then the following is true for any $\delta\in (0,\log 2/e)$:
\[\Pr\Big(|\hat{\theta}_i - \theta^*_i| > B_{i,\delta} \Big)\leq 2\Big(1 + \frac{24}{(\log 2)^2}\Big)\delta^2\]
	where the number $B_{i,\delta} > 0$ is defined as in \eqref{eq:boundterm}.
\end{thm}

\begin{remark}
	We bring attention to the fact that Theorem~\ref{thm:mean_elementwise} is clean, holds for all entries $i \in [1:n]$ and is easily interpretable. 
	\end{remark}

As before, summing up the pointwise bound above for all $1\leq i\leq n$, we have the following bound on the sum of squared errors:
\begin{cor}[Sum of squared estimation errors for mean regression]\label{cor:mean_estimationerror}
Suppose \textbf{Assumption Q3} holds. Then the sum of squared errors is bounded above with high probability: for all $\delta\in (0,n(\log2/e)^2)$,
\begin{align}\nonumber
&\Pr\Bigg(\sum_{i=1}^n|\hat{\theta}_i -\theta_i^*|^2\leq 192\sigma^2\Big(\log\log 2n+\frac{1}{2}\log\frac{n}{\delta}\Big)\Big(K+\sum_{k=1}^K\log\frac{m_k}{2}\Big)\\&+ 24n\sigma^4\Big(\frac{\log\log^2(2n) + \frac{1}{4}\log ^2\frac{n}{\delta}}{\lambda^2}\Big)+ 12K\sigma^2\log\frac{n}{\delta}+24\lambda^2\sum_{k=1}^K\frac{1}{m_k}\Bigg)\geq 1-4\Bigg(1 + \frac{24}{(\log 2)^2}\Bigg)\delta\label{eq:estimationerror2_mean}.
\end{align}  
\end{cor}
\begin{proof}
Following the proof of \eqref{eq:sumofsquaredquantile1}, 
\begin{align}
&\sum_{i=1}^nB_{i,\delta}^2\leq 
192\sigma^2(\log\log 2n+\log\frac{1}{\delta})(K+\sum_{k=1}^K\log\frac{m_k}{2})\\&+ 24n\sigma^4\left(\frac{\log\log^2(2n) + \log ^2\frac{1}{\delta}}{\lambda^2}\right)+ 24K\sigma^2\log\frac{1}{\delta}+24\lambda^2\sum_{k=1}^K\frac{1}{m_k}.\label{eq:sumofsquaredmean}
\end{align}
holds with probability $1-n\Big(1 + \frac{24}{(\log 2)^2}\Big)\delta^2$.
Replacing $\delta$ with $\sqrt{\delta/n}$, the corollary is proved.
\end{proof}

\begin{remark}
For ease of readability, we can again consider the main case of interest, when $\delta=O(1)$ and the length of all intervals are in the same order, i.e.,  $m_1=\cdots=m_K=\Theta(n/K)$, \eqref{eq:estimationerror2_mean}  becomes
\begin{equation}\label{eq:estimationerror3_mean}
\sum_{i=1}^n|\hat{\theta}_i -\theta_i^*|^2\leq O\left(K\sigma^2 \log n\log\frac{n}{K}+n\sigma^4\frac{\log^2n}{\lambda^2}+\frac{\lambda^2K^2}{n}\right). 
\end{equation}

Furthermore, when we set $\lambda=\log n\sqrt{n/K}$, we obtain
\begin{align}\label{eq:estimationerror4_mean}
\sum_{i=1}^n|\hat{\theta}_i-\theta_i^*|^2\leq O\left(K \log n\log\frac{n}{K}\right)\end{align}

\end{remark}

\subsection{Comparison with existing works on the estimation errors}\label{sec:discuss2}


Here we do a quick literature survey of existing results on the sum of squared error of fused lasso under the setting that $\{\epsilon_i\}_{i=1}^n$ are i.i.d. sampled from $N(0,\sigma^2)$. As mentioned before, to the best of our knowledge pointwise error bounds were not available before this work. 
\begin{itemize}
\item \cite{Lin2016,Lin20162} assume that the noises $\{\epsilon_i\}_{i=1}^n$ are i.i.d. sampled from a sub-Gaussian distribution with parameter $\sigma=1$. In \cite[Corollary 1]{Lin20162}, they show that for $\lambda=(n m_{\min})^{1/4}$, there exists constants $c, C, N$ that only depend on $\sigma$  such that for all $\gamma>1$ and $n>N$,
\[
\Pr\left(\sum_{i=1}^n|\hat{\theta}_i -\theta_i^*|^2\leq \gamma^2{cK}\left((\log K+\log\log n)\log n+{\frac{\lambda^2}{m_{\min}}}\right)\right)>1-\exp(-C\gamma).
\]

In the supplement of Lin et al. \cite[(A.10)]{Lin20162}, it is proved that for sufficiently large $n$, there exist constants $C, c$ such that with probability at least $1-\exp(-C\gamma)$, \begin{equation}\label{eq:lin}
\sum_{i=1}^n|\hat{\theta}_i -\theta_i^*|^2\leq\gamma^2 cK\left((\log K+\log\log n)\log n+{\frac{\lambda^2}{m_{\min}}}+\frac{n}{\lambda^2}\right)
\end{equation}

\item Guntuboyina et al. \cite{Guntuboyina2017} analyze Fused Lasso along with trend higher order versions known as Trend Filtering, with the assumption that the noises are i.i.d. sampled from the Gaussian distribution $N(0,\sigma^2)$. In \cite[Corollary 2.8]{Guntuboyina2017}, they show
\begin{equation}\label{eq:Guntuboyina2017}
\Expect \sum_{i=1}^n|\hat{\theta}_i -\theta_i^*|^2 \leq C\sigma^2\left( n\Delta_1+{(\lambda-\lambda^*)^2}\alpha_m\right)
\end{equation}
for every $\lambda\geq \lambda^*$, where  $\lambda^*$ is chosen implicitly in \cite[(27)]{Guntuboyina2017} and
\[
\Delta_1=\frac{K}{n}\log\left(\frac{en}{K}\right)+\frac{\alpha_m}{K}\log\left(\frac{en}{K}\right)+\frac{\sqrt{\alpha_m}}{\sqrt{n}}, 
\]
where $\alpha_m=\sum_{\eta_{k-1}\neq \eta_k}\frac{1}{m_k}$.


\item Ortelli and van de Geer \cite[Corollary 5.6]{Ortelli2018} improve the bounds in \cite{dalalyan2017}:  when  $\lambda=O(\sigma\sqrt{2n\log (n/\delta)})$, then with probability $1-\delta$, the estimation error is bounded above by
\begin{equation}\label{eq:Ortelli2018}
\sum_{i=1}^n|\hat{\theta}_i -\theta_i^*|^2\leq O\left(\sigma^2{K\log(n/\delta)}\left(\log \frac{n}{K}+\frac{n}{m_H}\right)\right),
\end{equation}
where the $m_H=\frac{K}{\sum_{k=1}^K\frac{1}{m_k}}$ is the harmonic mean between the distances of jumps, i.e., the harmonic mean of $n_{k+1}-n_k$ for all $1\leq k\leq K$.

\end{itemize}

\noindent\textbf{Comparisons}
\begin{itemize}
\item Compared with \cite{Ortelli2018}, our rate  \eqref{eq:estimationerror4_mean} is smaller by removing the term $n/m_H$. 

\item Compared with the rate in \cite{Lin2016,Lin20162} and \cite{Guntuboyina2017}, our rate for the squared error is worse by a factor of $\frac{\log\frac{n}{K}}{\log K+\log\log n}$ and $\log n$ respectively. The reason is that here we are aiming for point wise bounds and then summing up these pointwise bounds to arrive at a bound for the squared error. 
\item Our bound depends cleanly and explicitly on the tuning parameter $\lambda$ for a large range of $\lambda$, while the results of \cite{Guntuboyina2017} and \cite{Ortelli2018} only hold under an implicit optimal choice of $\lambda$.
\item We establish a truly nonasymptotic explicit bound in \eqref{eq:estimationerror2_mean}, while most of the existing works are asymptotic in the sense that they depend implicitly on unspecified constants or other unknown problem parameters. 
\item The main probabilistic part of our proof uses a nonasymptotic law of iterated logarithm. Instead, one could also use a cruder argument which applies Hoeffding's inequality and a union bound for means of the error random variables over all possible intervals. This would give slightly worse log factors. Upon observing this, it is not hard to see that our proof technique can be easily generalized to the setting where $\sigma_i$ are i.i.d. sampled from a non-subgaussian distribution, such as a zero-mean subexponential distribution. In this case, our estimation on the sum of squared errors would be larger by an order of at most $\log^2 n$. Existng proof techniques such as in~\cite{Guntuboyina2017},\cite{Lin2016,Lin20162} and~\cite{Ortelli2018} are not easily extendable to sub exponential noises to the best of our understanding. 
\end{itemize}

\noindent\textbf{An improved bound on the sum of squared errors} Similar to the discussion in Section~\ref{sec:discuss1}, we may replace the term $24\lambda^2\sum_{k=1}^K\frac{1}{m_k}$  in \eqref{eq:estimationerror2_mean} with \[144\lambda^2\sum_{\eta_{k-1}\neq \eta_k}\frac{1}{m_k}.\] As explained before, this makes the minimum length assumption, required for our bound to be meaningful, less stringent. 

\section{Discussion}
We first summarize our contributions in this paper. We analyzed the fused lasso estimator with a general convex loss function \eqref{eq:optimization} and established an element-wise upper bound of the estimation error for the first time. The main advantage of our result is its elementwise and nonasymptotic nature, and the fact that it can be applied for a general class of convex functions $\rho$. The derived elementwise bound imply tight bounds for a global loss as well: it improves the existing optimal result on quantile fused lasso and recoveres the exising optimal result on fused lasso up to a logarithmic factor.

Our work here raises a few intriguing follow up questions. The analysis presented here is valid under a fixed choice of the tuning parameter $\lambda.$ It will be interesting if one can obtain a data-driven strategy of choosing $\lambda$ and still attain these elementwise error bounds. Our current element wise error bounds are in terms of some unknown signal parameters. A natural question is whether these signal parameters can be estimated which would give a way to obtain finite sample confidence bands for the underlying function. Such elementwise confidence bands are important in application domains such as contextual bandits, see~\cite{chatterjee2021regret}.


Our proof technique to analyze Fused Lasso is new. Our main insight is that we can analyze each constant piece of the true signal separately. We hope that our proof technique can be useful for several natural extensions of the Fused Lasso estimator, including higher order trend filtering estimators for which pontwise error bounds are not known yet. 


It would also be interesting to examine other versions of the Fused Lasso estimator such as  2D total variation denoising \citep{pmlr-v49-huetter16,chatterjee2021new}, and the graph fused lasso~\citep{Hallac2015,Tansey2015,barberoTV14}.  For these extensions, all existing analyses bound a global loss. We hope and expect that our main strategy of ``reduction to one piece'' can be applied to establish a nonasymptotic elementwise error bound for these more involved settings as well. We leave this for future work.

\section{Appendix}
\subsection{Technical proofs}
\begin{proof}[Proof of Corollary~\ref{cor:quantile2}]
(a) If $\hat{\theta}_i -\theta_i^*> \frac{B_{i,\delta}^{\text{quantile}}}{{L}}$, then
\[
\Pr\big(0\leq \epsilon < \hat{\theta}_i - \theta^*_i\big)\geq \Pr\Big(0\leq \epsilon \leq  \frac{B_{i,\delta}^{\text{quantile}}}{{L}}\Big) = F\Big( \frac{B_{i,\delta}^{\text{quantile}}}{{L}}\Big)-F(0)\geq  {B_{i,\delta}^{\text{quantile}}},
\]
where the equality follows from the assumption of $\Pr(\epsilon=0)=0$ and the last inequality follows from \eqref{eq:q2}.

Similarly, if $\hat{\theta}_i -\theta_i^*<-\frac{B_{i,\delta}^{\text{quantile}}}{{L}}$, then
\[
\Pr\big( \hat{\theta}_i - \theta^*_i<\epsilon\leq 0\big)\geq F(0)-F\Big(- \frac{B_{i,\delta}^{\text{quantile}}}{{L}}\Big)\geq  {B_{i,\delta}^{\text{quantile}}}.
\]
Combining the above equations with Theorem~\ref{thm:quantile},  Corollary~\ref{cor:quantile2} (a) is proved.\\

\end{proof}

\begin{proof}[Proof of Lemma~\ref{lemma:assumption}]
The first assumption \eqref{eq:quantile2a_assumption} is satisfied when
\[
\max\left(2\sqrt{\frac{\log\log 2\max(3,d_i)}{{\max(3,d_i)}}},2\sqrt{\frac{\log \frac{1}{\delta}}{d_i}}, \frac{\log\log(2m_{k(i)}) }{\lambda}, \frac{\log \frac{1}{\delta}}{\lambda} ,\frac{\sqrt{m_{k(i)}\log \frac{1}{\delta}}}{m_{k(i)}}, \frac{2\lambda}{m_{k(i)}}\right)\leq \frac{L}{6}.
\]
Note that for $x\geq 3$, $g(x)=\frac{x}{\log\log 2x}$ satisfies $g(x)\geq \sqrt{x}$, the above equation and \eqref{eq:quantile2a_assumption} is satisfied when  
\begin{align*}
d_i\geq \max\left(3,\frac{12^4}{L^4}, \frac{12^2}{L^2}\log\frac{1}{\delta}\right), 
\frac{6(\log\log (2m_{k(i)})+\log\frac{1}{\delta})}{L}\leq\lambda\leq \frac{m_{k(i)}L}{12}, m_{k(i)}\geq \frac{36\log\frac{1}{\delta}}{L^2}
\end{align*}

Similarly, the second assumption \eqref{eq:quantile2b_assumption} is independent of the location $i$, and is satisfied when
\[
n\geq \frac{4}{L^2}\log \frac{1}{\delta}, \,\,\,\,\lambda\geq \frac{2}{L}(\log\log(2n)+\log \frac{1}{\delta}).
\]
\end{proof}

\begin{proof}[Proof of Corollary~\ref{cor:estimationerror2}]
We first present a lemma that establishes a uniform upper bound and lower bound of $\{\hat{\theta}_i\}_{i=1}^n$. The proof is based on an argument similar to that of Lemma~\ref{lemma:intermediate}  and Theorem~\ref{thm:general}.
\begin{lemma}\label{lemma:uniformbound} Suppose that Assumption Q2 holds. 
 If\begin{align}\label{eq:quantile2b_assumption}
	B_{\text{uniform},\delta}^{\text{quantile}}&:=\frac{\log\log(2n)+\log \frac{1}{\delta}}{\lambda}+\sqrt{\frac{\log \frac{1}{\delta}}{n}}\leq{L},\end{align}  then the following is true for any $\delta\in (0,\log 2/e)$:
\begin{align}\label{eq:quantile2b}
&Pr\left(\min_{1\leq i\leq n}\theta^*_i-\frac{1}{{L}}B_{\text{uniform},\delta}^{\text{quantile}}\leq \min_{1\leq i\leq n}\hat{\theta}_i\leq \max_{1\leq i\leq n}\hat{\theta}_i\leq \max_{1\leq i\leq n}\theta^*_i+\frac{1}{{L}}B_{\text{uniform},\delta}^{\text{quantile}}\right)\\\geq& 1-2\Big(1 + \frac{24}{(\log 2)^2}\Big) \delta^2.\nonumber\end{align}
\end{lemma}
\begin{proof}[Proof of Lemma~\ref{lemma:uniformbound}]
Applying the same argument as in Lemma~\ref{lemma:intermediate} to the problem \eqref{eq:optimization}, we have that for any $i \in [1:n]$ and any $\alpha \geq 0$, 
 \begin{align}\label{eq:intermediate1}
\Big\{\tilde{\theta}_i\geq \alpha \Big\}&\subseteq \Big\{\exists s,t: 1\leq s\leq i\leq t\leq n, z_3(s,t)\geq 0 \Big\}\\
\Big\{\tilde{\theta}_i\leq \alpha \Big\}&\subseteq \Big\{\exists s,t: 1\leq s\leq i\leq t\leq n, z_4(s,t)\leq 0 \Big\}\label{eq:intermediate2}.
 \end{align}
 where 
 \[
 z_3(s,t)= \begin{cases}\sum_{j=s}^t\rho'_{+}(y_j-\alpha)-2\lambda,\,\,&\text{if $s\neq 1$ and $t\neq n$}\\\sum_{j=s}^t \rho'_{+}(y_j-\alpha)-\lambda ,\,\,&\text{if $s\neq 1, t=n$ or $s=1,t\neq n$}\\
\sum_{j=s}^t\rho'_{+}(y_j-\alpha),\,\,&\text{if $s =1$ and $t= n$}\end{cases}\]
and
 \[
 z_4(s,t)= \begin{cases}\sum_{j=s}^t \rho'_{-}(y_j-\alpha)+2\lambda,\,\,&\text{if $s\neq 1$ and $t\neq n$}\\\sum_{j=s}^t\rho'_{-}(y_j-\alpha)+\lambda,\,\,&\text{if $s\neq 1, t=n$ or $s=1,t\neq n$}\\
\sum_{j=s}^t\rho'_{-}(y_j-\alpha),\,\,&\text{if $s =1$ and $t= n$}\end{cases}\]
 We can now write for any $i \in [1:n]$ and any $\alpha > 0$,
\begin{equation}\label{eq:p1234_2}
\Pr\Big(\hat{\theta}_i > \alpha \Big)\leq p_1'+p_2'+p_3'+p_4',
 \end{equation}
 where $p_1',p_2',p_3',p_4'$ are probabilities of events given by 
 \begin{align*}
 p_1'=&\Pr\Big(\exists s\neq 1, t\neq m: s\leq i\leq t,\,\,  z(s,t)\geq 2\lambda\Big)=p_1\\
 p_2'=&\Pr\Big(\exists t\geq i: z(1,t)\geq \lambda\Big)\\
p_3'=&\Pr\Big(\exists s\leq i: z(s,m)\geq \lambda\Big)\\
p_4'=&\Pr\Big(z(1,m)\geq 0 \Big).
\end{align*} 

A similar argument as in Proposition~\ref{thm:general} shows that \begin{equation}\Pr(\hat{\theta}_i\geq \max_{i=1,\cdots,n}{\theta^*_i}+\alpha)\leq (1+\frac{24}{(\log 2)^2})\delta^2,\label{eq:thm2b}\end{equation} if $\alpha>0$ is chosen such that
\begin{align}\label{eq:T1234_2}
&\Expect \rho'_-(\epsilon_1-\alpha)\leq  -\max(T_1', T_2', T_3', T_4')
\end{align}
for 
\begin{align*}
&T_1'=\max_{i\leq t\leq n}\left(4\sigma\sqrt{\frac{\log\log(2t)+\log\frac{1}{\delta}}{t}}-\frac{\lambda}{t}\right)\leq \max_{i\leq t\leq m}\left(4\sigma\sqrt{\frac{\log\log(2n)+\log\frac{1}{\delta}}{t}}-\frac{\lambda}{t}\right)\\\leq&4\sigma^2\frac{\log\log(2n)+\log\frac{1}{\delta}}{\lambda},
\end{align*}
 similarly,
\[
T_2'=\max_{1\leq s\leq i}\left(4\sigma\sqrt{\frac{\log\log(2(n-s+1))+\log\frac{1}{\delta}}{n-s+1}}-\frac{\lambda}{n-s+1}\right)\leq 4\sigma^2\frac{\log\log(2n)+\log\frac{1}{\delta}}{\lambda},
\]
$T_3'\leq 4\sigma^2\frac{\log\log(2n)+\log\frac{1}{\delta}}{\lambda}$ (similar to the proof of $T_3$), 
and
\[
T_4'=\frac{\sqrt{-4n\sigma^2\log \delta}}{n}=2\sigma\sqrt{\frac{\log\frac{1}{\delta}}{n}}.
\]
Plug in $\sigma=1/2$, the upper bound in part (b) is then proved by combining \eqref{eq:thm2b}, \eqref{eq:T1234_2}, and the estimations of $T_1'$, $T_2'$, $T_3'$, $T_4'$. The lower bound would be proved similarly.
\end{proof}
The assumption \eqref{eq:quantile2b_assumption} in Lemma~\ref{lemma:uniformbound} is satisfied when 
\begin{equation}\label{eq:quantile2b_assumption2}
n\geq \frac{4}{L^2}\log \frac{1}{\delta}, \,\,\,\,\lambda\geq \frac{2(\log\log(2n)+\log \frac{1}{\delta})}{L}.
\end{equation}
As a result, if $n$ is large and $\lambda$ is well-chosen, this assumption holds. 

Combining Corollary~\ref{cor:quantile2} with $\delta=\sqrt{\delta_0/n}$, Lemma~\ref{lemma:uniformbound} with $\delta=\sqrt{\delta_0}$,  Lemma~\ref{lemma:assumption}, and \eqref{eq:quantile2b_assumption2}, we have that when
\[
m_{\min}\geq\frac{36}{L^2}\log\frac{1}{\sqrt{\delta_0/n}}, \,\, \frac{6(\log\log (2n)+\log\frac{1}{\sqrt{\delta_0/n}})}{L}\leq\lambda\leq \frac{L}{12}m_{\min},
\]
then for all $\sqrt{\delta_0/n}\in (0,\log2/e)$, 
\[
B_{\text{uniform},\sqrt{\delta_0}}^{\text{quantile}}\leq B_{\text{uniform},\sqrt{\delta_0/n}}^{\text{quantile}}\leq L.
\]
As a result, for all $\sqrt{\delta_0}\in (0,\log2/e)$, \eqref{eq:quantile2b} holds with $\delta=\sqrt{\delta_0}$.

In addition,  for $\calI=\{1\leq i\leq n: {B_{i,\sqrt{\delta_0/n}}^{\text{quantile}}}>{{L}}\}$
 \[|\calI|\leq 2K\max\left(3,\frac{12^4}{L^4}, \frac{12^2}{L^2}\log\frac{1}{\sqrt{\delta_0/n}}\right).\]
 Comining these estimations with \eqref{eq:quantile2a} and \eqref{eq:quantile2b}, we have that for all $\sqrt{\delta_0}\in (0,\log2/e)$,
\begin{equation}\label{eq:quantile3a}
\Pr\left(\sum_{i=1}^n|\hat{\theta}_i -\theta_i^*|^2\leq \sum_{i=1}^n\Big(\frac{B_{i,\sqrt{\delta_0/n}}^{\text{quantile}}}{{L}}\Big)^2 + |\calI|\Big(\frac{B_{\text{uniform},\sqrt{\delta_0}}^{\text{quantile}}}{{L}}+V\Big)^2\right)\geq 1-4\Big(1 + \frac{24}{(\log 2)^2}\Big)\delta_0.
\end{equation}
Considering that $|\calI|\leq n$ and ${B_{\text{uniform},{\delta}}^{\text{quantile}}}$ is a nonincreasing function of $\delta$, a relaxation of \eqref{eq:quantile3a} is 
\begin{equation}\label{eq:quantile3b}
\Pr\left(\!\sum_{i=1}^n|\hat{\theta}_i -\theta_i^*|^2\!\leq \sum_{i=1}^n\Big(\frac{B_{i,\sqrt{\delta_0/n}}^{\text{quantile}}}{{L}}\Big)^2 \!+\! 2n\Big(\frac{B_{\text{uniform},\sqrt{\delta_0/n}}^{\text{quantile}}}{{L}}\Big)^2\!+\!2|\calI|V^2\!\right)\!\!\geq \!1-4\Big(\!1 + \frac{24}{(\log 2)^2}\!\Big)\delta_0.
\end{equation}

It remains to estimate $\sum_{i=1}^n\Big(B_{i,\delta}^{\text{quantile}}\Big)^2$. Using $(a_1+\cdots+a_6)^2\leq 6(a_1^2+\cdots+a_6^2)$, it is bounded above by
\begin{equation}\label{eq:bound_B1}
24\left({\frac{\log\log 2\max(3,d_i)}{{\max(3,d_i)}}}+{\frac{\log \frac{1}{\delta}}{d_i}}\right) +  6\left(\frac{\log\log^2(2m_{k(i)}) + \log ^2\frac{1}{\delta}}{\lambda^2}\right) +6\frac{\log \frac{1}{\delta}}{m_{k(i)}}+24 \frac{\lambda^2}{m_{k(i)}^2}
\end{equation}
Applying $\sum_{i=1}^n\frac{1}{n}\leq \ln n+1$, we have $\sum_{i=1}^n\frac{1}{d_i}\leq 2\sum_{k=1}^K(\log\frac{m_k}{2}+1)$. As a result, the summation of the first term in \eqref{eq:bound_B1} over $i=1,\cdots,n$ can be bounded by
\[
48(\log\log 2n+\log\frac{1}{\delta})(K+\sum_{k=1}^K\log\frac{m_k}{2}).
\] 
The summation of the third term is $6K\log\frac{1}{\delta}$, and the summation of the fourth term is $24\lambda^2\sum_{k=1}^K\frac{1}{m_k}$. Combining these estimations, we have
\begin{align}
&\sum_{i=1}^n\Big(B_{i,\delta}^{\text{quantile}}\Big)^2\leq 48(\log\log 2n+\log\frac{1}{\delta})(K+\sum_{k=1}^K\log\frac{m_k}{2})\\&+ 6n\left(\frac{\log\log^2(2n) + \log ^2\frac{1}{\delta}}{\lambda^2}\right)+ 6K\log\frac{1}{\delta}+24\lambda^2\sum_{k=1}^K\frac{1}{m_k}.\label{eq:sumofsquaredquantile1}
\end{align}
On the other hand, 
\begin{align*}
&\Big(B_{\text{uniform},\delta}^{\text{quantile}}\Big)^2=\Big(\frac{\log\log(2n)+\log \frac{1}{\delta}}{\lambda}+\sqrt{\frac{\log \frac{1}{\delta}}{n}}\Big)^2\\
\leq &3\frac{\log\log^2(2n)}{\lambda^2}+3\frac{\log^2\frac{1}{\delta}}{\lambda^2}+3\frac{\log\frac{1}{\delta}}{n}.
\end{align*}
Combining the estimations above with \eqref{eq:quantile3b} and plug in $\delta=\sqrt{\delta_0/n}$, we reach the corollary (in the statement we write $\delta$ instead of $\delta_0$).

\end{proof}

\subsection{Improved estimation on the sum of squared errors}\label{sec:improved2}
Following the proof of Theorem~\ref{thm:quantile} and Corollary~\ref{cor:quantile2}, the term $\frac{2\lambda}{m_{k(i)}}$ in $B_{i,\delta}^{\text{quantile}}$ comes from the probability $p_4$ in the proof of Lemma~\ref{lemma:intermediate}. Let $k(i)'$ and $k(i)''$ chosen to be the left-most and right-most indices such that $[\theta^*_{n_{k(i)'}},\cdots,\theta^*_{n_{k(i)}},\cdots, \theta^*_{n_{k(i)''}}]$ is a  monotonic sequence, then by the properties of $k(i)'$ and $k(i)''$, we may consider the intervals $[{n_{k(i)'}},n_{k(i)+1}-1]$ and $[n_{k(i)},n_{k(i)''+1}-1]$ separately for the estimation of the upper bounds and lower bounds of $\hat{\theta}_i$ (instead of the interval of the piecewise constant observations $[n_{k(i)},n_{k(i)+1}-1]$ in the proof of Lemma~\ref{lemma:intermediate}). 

Then, the term $\frac{2\lambda}{m_{k(i)}}$ in $B_{i,\delta}^{\text{quantile}} $ can be replaced with \begin{equation}\label{eq:bound2_relax}
2\left(\frac{\lambda}{m_{k(i)}^{\mathrm{left}}}+\frac{\lambda}{m_{k(i)}^{\mathrm{right}}}\right), 
\end{equation}
where
\[
m_{k(i)}^{\mathrm{left}}=m_{k(i)'}+m_{k(i)'+1}+\cdots+m_{k(i)}, \,\,\,\,m_{k(i)}^{\mathrm{right}}=m_{k(i)}+m_{k(i)+1}+\cdots+m_{k(i)''}.
\]

As a result, the term $\sum_{k=1}^K\frac{\lambda^2}{m_k}$ in \eqref{eq:estimationerror2}, which is obtained from
\begin{equation}\label{eq:improvement1}
\sum_{i=1}^n\left(\frac{\lambda}{m_{k(i)}}\right)^2
\end{equation}
 can be replaced with 
\begin{equation}\label{eq:improvement2}
\sum_{i=1}^n\left(\frac{\lambda}{m_{k(i)}^{\mathrm{left}}}+\frac{\lambda}{m_{k(i)}^{\mathrm{right}}}\right)^2\leq 2\sum_{i=1}^n\left(\left(\frac{\lambda}{m_{k(i)}^{\mathrm{left}}}\right)^2+\left(\frac{\lambda}{m_{k(i)}^{\mathrm{right}}}\right)^2\right).
\end{equation}

To simplify \eqref{eq:improvement2}, let us introduce an auxillary Lemma:
\begin{lemma}\label{lemma:iterativesum}
For any integer $K\geq 1$ and any $m_1,\cdots,m_K>0$, we have
\begin{equation}\label{eq:iterativesum}
\frac{m_1}{(m_1+\cdots+m_K)^2}+\frac{m_2}{(m_2+\cdots+m_K)^2}+\cdots+\frac{m_{K-1}}{(m_{K-1}+m_K)^2}+\frac{m_K}{m_K^2}\leq \frac{3}{m_k}.
\end{equation}
\end{lemma}
\begin{proof}[Proof of Lemma~\ref{lemma:iterativesum}]
The proof is based on induction. First, the first term of the LHS of \eqref{eq:iterativesum} has the upper bound of $\frac{m_1}{(m_1+\cdots+m_K)^2}\leq \frac{2}{(m_2+\cdots+m_K)}$. Second, we have the following bound on the sum of the first two terms of the LHS of \eqref{eq:iterativesum}:
\begin{align*}
&\frac{2}{(m_2+\cdots+m_K)}+\frac{m_2}{(m_2+\cdots+m_K)^2}\leq \frac{1}{m_3+\cdots+m_K}\min_{x}\left(\frac{2}{(x+1)}+\frac{x}{(x+1)^2}\right)\\\leq &\frac{2}{m_3+\cdots+m_K}.
\end{align*}
Similarly, the first three terms of the LHS of \eqref{eq:iterativesum} is bounded by $\frac{2}{m_4+\cdots+m_K}$. Repeat the procedure,   Lemma~\ref{lemma:iterativesum} is proved.
\end{proof}

As for the RHS of \eqref{eq:improvement2} squared over all $1\leq i\leq n$, we define $1\leq k_1< \cdots\leq k_l< K$ to be the locations of the ``change of directions'' in the sense that
\[
\eta_1=\cdots=\eta_{k_1}\neq \eta_{k_1+1}=\cdots=\eta_{k_2}\neq \eta_{k_2+1}=\cdots=\eta_{k_3}\neq \eta_{k_3+1}+\cdots.
\]

Then we have
\begin{align*}
\sum_{i=1}^n\left(\frac{1}{m_{k(i)}^{\mathrm{right}}}\right)^2\leq &\frac{m_1}{(m_1+\cdots+m_{k_1})^2}+\frac{m_2}{(m_2+\cdots,m_{k_1})^2}+\cdots+\frac{m_{k_1}}{(m_{k_1})^2}\\&+
\frac{m_{k_1+1}}{(m_{k_1+1}+\cdots+m_{k_2})^2}+\frac{m_{k_1+2}}{(m_{k_1+2}+\cdots+m_{k_2})^2}+\cdots+\frac{m_{k_2}}{(m_{k_2})^2}+\cdots\\\leq &3\left(\frac{1}{m_{k_1}}+\frac{1}{m_{k_2}}+\cdots+\frac{1}{m_{k_l}}+\frac{1}{m_{K}}\right),
\end{align*}
and similarly,
\[
\sum_{i=1}^n\left(\frac{1}{m_{k(i)}^{\mathrm{right}}}\right)^2\leq  3\left(\frac{1}{m_{1}}+\frac{1}{m_{k_1+1}}+\frac{1}{m_{k_2+1}}+\cdots+\frac{1}{m_{k_l+1}}\right).
\]
As a result, the RHS of \eqref{eq:improvement2} is bounded by
\[
6\left(\frac{1}{m_1}+(\frac{1}{m_{k_1}}+\frac{1}{m_{k_1+1}})+\cdots+(\frac{1}{m_{k_l}}+\frac{1}{m_{k_l+1}})+\frac{1}{m_{K}}\right)=6\sum_{1\leq k\leq K: \eta_{k-1}\neq \eta_k}\frac{1}{m_k}\]
and the improved bound is obtained by combining it with \eqref{eq:improvement1} and \eqref{eq:improvement2}.

\vskip 0.2in

\bibliography{bib-online}

\end{document}